\documentclass[12pt]{article}
 \usepackage{amsmath}
 \usepackage{amssymb}
\usepackage{amsfonts}
\usepackage{amsbsy}
\newenvironment{proof}[1][Proof]{\begin{trivlist}
\item[\hskip \labelsep {\bfseries #1}]}{\end{trivlist}}
    \newcommand{\qed}{\nobreak \ifvmode \relax \else
          \ifdim\lastskip<1.5em \hskip-\lastskip
          \hskip1.5em plus0em minus0.5em \fi \nobreak
          \vrule height0.75em width0.5em depth0.25em\fi}
\newtheorem{myth}{Theorem}[section]
\newtheorem{mylem}{Lemma}[section]
\newtheorem{mycor}{Corollary}[section]

\newtheorem{mydef}{Definition}[section]

\newtheorem{myrem}{Remark}[section]
\newtheorem{myexam}{Example}[section]
\begin{document}
\date{On the 50$^\textrm{th}$ Anniversary of Obafemi Awolowo University}
\title{Holomorph of generalized Bol loops
}
\author{J. O. Ad\'en\'iran \\
Department of Mathematics,\\
Federal University of Agriculture, \\
Abeokuta 110101, Nigeria.\\
ekenedilichineke@yahoo.com\\
adeniranoj@unaab.edu.ng
\and
 T. G. Jaiy\'e\d ol\'a \footnote{corresponding author}\\
Department of Mathematics,\\
Obafemi Awolowo University,\\
Ile Ife 220005, Nigeria.\\
jaiyeolatemitope@yahoo.com\\tjayeola@oauife.edu.ng
\and
K. A. \`{I}d\`{o}w\'{u} \\
Department of Mathematics,\\
Federal University of Agriculture, \\
Abeokuta 110101, Nigeria.\\
idowuka@hotmail.com}
\maketitle
\begin{abstract}
The notion of the holomorph of a generalized Bol loop and generalized flexible-Bol loop are characterized. With the aid of two self-mappings on the holomorph of a loop, it is shown that: the loop is a generalized Bol loop if and only if its holomorph is a generalized Bol loop; the loop is a generalized flexible-Bol loop if and only if its holomorph is a generalized flexible-Bol loop. Furthermore, elements of the Bryant Schneider group of a generalized Bol loop are characterized in terms of pseudo-automorphism, and the automorphisms gotten are used to build the holomorph of the generalized Bol loop.
\end{abstract}
\section{Introduction}
\paragraph{}The birth of Bol loops can be traced back to Gerrit Bol \cite{16} in 1937 when he established the relationship between Bol loops and Moufang loops, the latter which was discovered by Ruth Moufang \cite{19}. Thereafter, a theory of Bol loops evolved through the Ph.D. thesis of Robinson \cite{25} in 1964 where he studied the algebraic properties of Bol loops, Moufang loops and Bruck loops, isotopy of Bol loop and some other notions on Bol loops. Some later results on Bol loops and Bruck loops can be found in \cite{1,2}, \cite{6,16,7,8}, \cite{bur}, \cite{30,31} and \cite{38}

In the 1980s, the study and construction of finite Bol loops caught the attention of many researchers among whom are Burn \cite{bur,bur1,bur2}, Solarin and Sharma \cite{34,32.1,35} and others like Chein and Goodaire \cite{gg22,chein2,chein1}, Foguel at. al. \cite{phd127}, Kinyon and Phillips \cite{phd115,phd121} in the present millennium. One of the most important results in the theory of Bol loops is the solution of the open problem on the existence of a simple Bol loop which was finally laid to rest by Nagy \cite{nagy1,nagy2,nagy3}.

In 1978, Sharma and Sabinin \cite{32,33} introduced and studied the algebraic properties of the notion of half-Bol loops(left B-loops). Thereafter, Adeniran \cite{phd111}, Adeniran and Akinleye \cite{1}, Adeniran and Solarin \cite{3} studied the algebraic properties of generalized Bol loops. Also, Ajmal \cite{5} introduced and studied the algebraic properties of generalized Bol loops and their relationship with M-loops (cf. identity \eqref{eq:8.3}).

Interestingly, Adeniran \cite{phd79}, \cite{7}, \cite{phd40}, \cite{phd80}, \cite{phd44}, \cite{25,phd7} are devoted to study the holomorphs of Bol loops,
conjugacy closed loops, inverse property loops, A-loops, extra loops, weak inverse property loops and Bruck loops.

The Bryant-Schneider group of a loop was introduced by Robinson \cite{phd93}, based on the motivation of \cite{phd92}. Since the advent of the
Bryant-Schneider group, some studies by Adeniran \cite{phd142} and
Chiboka \cite{phd94} have been done on it relative to CC-loops and
extra loops.

The objectives of this present work are to study the structure of the holomorph of a generalized Bol loop and generalized flexible Bol loop, and also to characterize elements of the Bryant-Schneider group of a generalized Bol loop (generalized flexible Bol loop) and use this elements to build the holomorph of a generalized Bol loop (generalized flexible Bol loop).

\section{Preliminaries}
\paragraph{}
Let $L$ be a non-empty set. Define a binary operation ($\cdot $) on
$L$ : If $x\cdot y\in L$ for all $x, y\in L$, $(L, \cdot )$ is
called a groupoid. If the equations:
\begin{displaymath}
a\cdot x=b\qquad\textrm{and}\qquad y\cdot a=b
\end{displaymath}
have unique solutions for $x$ and $y$ respectively, then $(L, \cdot
)$ is called a quasigroup. For each $x\in L$, the elements $x^\rho
=xJ_\rho\in L$ and $x^\lambda =xJ_\lambda\in L$ such that
$xx^\rho=e^\rho$ and $x^\lambda x=e^\lambda$ are called the right
and left inverse elements of $x$ respectively. Here, $e^\rho\in L$
and $e^\lambda\in L$ satisfy the relations $xe^\rho =x$ and
$e^\lambda x=x$ for all $x\in L$ and are respectively called the
right and left identity elements. Now, if $e^\rho=e^\lambda=e\in L$, then $e$ is called the identity element and $(L, \cdot )$ is called a loop. In case $x^\lambda =x^\rho$, then, we simply write $x^\lambda =x^\rho =x^{-1}=xJ$ and refer to $x^{-1}$ as the inverse of $x$.

Let $x$ be an arbitrarily fixed element in a loop $(G, \cdot )$. For any $y\in G$, the left and
right translation maps of $x\in G$, $L_x$ and $R_x$ are respectively
defined by
\begin{displaymath}
yL_x=x\cdot y\qquad\textrm{and}\qquad yR_x=y\cdot x
\end{displaymath}
A loop $(L,\cdot)$ is called
a (right) Bol loop if it satisfies the identity
\begin{equation}\label{eq:7}
(xy\cdot z)y=x(yz\cdot y)
\end{equation}
A loop $(L,\cdot)$ is called
a left Bol loop if it satisfies the identity
\begin{equation}\label{eq:7.1}
y(z\cdot yx)=(y\cdot zy)x
\end{equation}
A loop $(L,\cdot)$ is called
a Moufang loop if it satisfies the identity
\begin{equation}\label{eq:7.2}
(xy)\cdot (zx)=(x\cdot yz)x
\end{equation}
A loop $(L,\cdot)$ is called a right inverse property loop (RIPL) if $(L,\cdot)$ satisfies right inverse property (RIP)
\begin{equation}\label{eq:9}
(yx)x^{\rho}=y
\end{equation}
A loop $(L,\cdot)$ is called a left inverse property loop (LIPL) if $(L,\cdot)$ satisfies left inverse property (LIP)
\begin{equation}\label{eq:9.1}
x^\lambda(xy)=y
\end{equation}
A loop $(L,\cdot)$ is called an automorphic inverse property loop (AIPL) if $(L,\cdot)$ satisfies automorphic inverse property (AIP)
\begin{equation}\label{eq:10}
(xy)^{-1}=x^{-1}y^{-1}
\end{equation}

A loop $(L,\cdot)$ in which the mapping $x\mapsto x^2$ is a permutation, is called a Bruck loop if it is both a Bol loop and either AIPL or obeys the identity $xy^2\cdot x=(yx)^2$. (Robinson \cite{25})

Let $(L,\cdot)$ be a loop with a single valued self-map $\sigma:x\longrightarrow\sigma(x)$:

$(L,\cdot )$ is called a generalized (right) Bol loop or right B-loop if it satisfies
the identity
\begin{equation}\label{eq:8}
(xy\cdot z)\sigma(y)=x(yz\cdot\sigma(y))
\end{equation}
$(L,\cdot )$ is called a generalized left Bol loop or left B-loop if it satisfies
the identity
\begin{equation}\label{eq:8.1}
\sigma(y)(z\cdot yx)=(\sigma(y)\cdot zy)x
\end{equation}
$(L,\cdot )$ is called an M-loop if it satisfies
the identity
\begin{equation}\label{eq:8.3}
(xy)\cdot (z\sigma(x))=(x\cdot yz)\sigma(x)
\end{equation}

Let $(G,\cdot )$ be a groupoid(quasigroup, loop) and let $A,B$ and $C$ be three bijective
mappings, that map $G$ onto $G$. The triple $\alpha =(A,B,C)$ is
called an autotopism of $(G,\cdot )$ if and only if
\begin{displaymath}
xA\cdot yB=(x\cdot y)C~\forall~x,y\in G.
\end{displaymath}
Such triples form a group
$AUT(G,\cdot )$ called the autotopism group of $(G,\cdot )$.

If $A=B=C$, then $A$ is called an automorphism of the
groupoid(quasigroup, loop) $(G,\cdot )$. Such bijections form a
group $AUM(G,\cdot )$ called the automorphism group of $(G,\cdot )$.

The right nucleus of $(L,\cdot)$ is defined by $N_\rho (L,\cdot)=\{x\in L~|~zy\cdot x=z\cdot yx~\forall~y,z\in L\}$.

\begin{mydef}
Let $(Q,\cdot)$ be a loop and $A(Q)\le AUM(Q,\cdot)$
be a group of automorphisms of the loop $(Q,\cdot)$. Let $H=A(Q)\times Q$. Define $\circ$ on $H$ as
\begin{displaymath}
(\alpha,x)\circ(\beta,y)=(\alpha\beta,x\beta\cdot y)~\textrm{for all}~(\alpha,x),(\beta,y)\in H.
\end{displaymath}
$(H,\circ)$ is a loop and is called the A-holomorph of $(Q,\cdot)$.
\end{mydef}
The left and right translations maps of an element $(\alpha,x)\in H$ are respectively
denoted by $\mathbb{L}_{(\alpha,x)}$ and $\mathbb{R}_{(\alpha,x)}$.

\begin{myrem}
$(H,\circ)$ has a subloop $\{I\}\times Q$ that is isomorphic to $(Q,\cdot)$. As observed in Lemma 6.1 of Robinson \cite{25}, given a loop $(Q,\cdot)$ with an A-holomorph $(H,\circ )$, $(H,\circ )$ is a Bol loop if and only if $(Q,\cdot)$ is a $\theta$-generalized Bol loop for all $\theta\in A(Q)$. Also in Theorem 6.1 of Robinson \cite{25}, it was shown that $(H,\circ )$ is a Bol loop if and only if $(Q,\cdot)$ is a Bol loop and $x^{-1}\cdot x\theta\in N_\rho(Q,\cdot)$ for all $\theta\in A(Q)$.
\end{myrem}

\begin{mydef}
Let $(Q,\cdot)$ be a loop with a single valued self-map $\sigma$ and let $(H,\circ)$ be the A-holomorph of $(Q,\cdot)$ with single valued self-map $\sigma'$. $(Q,\cdot)$ is called a $\sigma$-flexible loop ($\sigma$-flexible) if
$$xy\cdot \sigma(x\delta)=x\cdot y\sigma(x\delta)~\textrm{for all}~x,y\in Q~\textrm{and some}~\delta\in A(Q).$$
$(H,\circ )$ is called a $\sigma'$-flexible loop ($\sigma'$-flexible) if
$$(\alpha,x)(\beta,y)\circ \sigma'(\alpha,x)=(\alpha,x)\circ (\beta,y)\sigma'(\alpha,x)~\textrm{for all}~(\alpha,x),(\beta,y)\in H.$$
If a loop is both a $\sigma$-generalised
Bol loop and a $\sigma$-flexible loop, then it is called a $\sigma$-generalised flexible-Bol loop.
\end{mydef}

If in this triple $(A,B,C)\in AUT(G,\cdot )$, $B=C=AR_c$, then A is called a pseudo-automorphism of a quasigroup $(G,\cdot )$ with companion $c\in G$. Such bijections form a
group $PS(G,\cdot )$ called the pseudo-automorphism group of $(G,\cdot )$.

\begin{mydef}\label{a1:3}(Robinson \cite{phd93})

Let $(G,\cdot )$ be a loop with symmetric group $SYM(G)$.
A mapping $\theta\in SYM(G)$ is called a
special map for $G$ if there exist $f,g\in G$ so that
$(\theta R_g^{-1},\theta L_f^{-1},\theta )\in AUT(G,\cdot )$.
\end{mydef}
\begin{myth}\label{a1:3l}(Robinson \cite{phd93})

Let $(G,\cdot )$ be a loop with symmetric group $SYM(G)$. The set of all special maps in $(G,\cdot )$ i.e.
$$BS(G,\cdot )=\{\theta\in SYM(G,\cdot )~:~\exists~f,g\in G~\ni~(\theta R_g^{-1},\theta L_f^{-1},\theta )\in AUT(G,\cdot
)\}$$ is a subgroup of $SYM(G)$ and is called the Bryant-Schneider group of the loop
$(G,\cdot )$.
\end{myth}

Some existing results on generalized Bol loops and generalized Moufang loops are highlighted below.
\begin{myth}(Adeniran and Akinleye \cite{1})\label{1}

If $(L,\cdot )$ is a generalized Bol
loop, then:
\begin{enumerate}
\item $(L,\cdot)$ is a RIPL.
\item $x^{\lambda}=x^{\rho}$ for all $x\in L$.
\item $R_{y\cdot \sigma(y)}=R_yR_{\sigma(y)}$ for all $y\in L$.
\item $[xy\cdot\sigma(x)]^{-1}=(\sigma(x))^{-1}y^{-1}\cdot x^{-1}$ for all $x,y\in L$.
\item $(R_{y^{-1}}, L_yR_{\sigma(y)},
R_{\sigma(y)}),(R_y^{-1}, L_yR_{\sigma(y)},
R_{\sigma(y)})\in AUT(L,\cdot )$ for all $y\in L$.
\end{enumerate}
\end{myth}
\begin{myth}(Sharma and Sabinin \cite{32})\label{2}

If $(L,\cdot )$ is a half Bol loop, then:
\begin{enumerate}
\item $(L,\cdot)$ is a LIPL.
\item $x^{\lambda}=x^{\rho}$ for all $x\in L$.
\item $L_{(x)}L_{(\sigma(x))} =
L_{(\sigma(x)x)}$ for all $x\in L$.
\item $(\sigma(x)\cdot yx)^{-1} = x^{-1}\cdot y^{-1} (\sigma(x))^{-1}$ for all $x,y\in L$.
\item $(R_{(x)}L_{(\sigma(x))},
L_{(x)^{-1}}, L_{(\sigma(x))}),(R_{(\sigma(x))}L_{(x)^{-1}}, L_{\sigma(x)}, L_{(x)^{-1}})\in AUT(L,\cdot )$ for all $x\in L$.
\end{enumerate}
\end{myth}
\begin{myth}(Ajmal \cite{5})\label{3}

Let $(L,\cdot )$ be a loop. The following statements are equivalent:
\begin{enumerate}
\item $(L,\cdot)$ is a M-loop;
\item $(L,\cdot)$ is both a left B-loop and a right B-loop;
\item $(L,\cdot)$ is a right B-loop and satisfies the LIP;
\item $(L,\cdot)$ is a left B-loop and satisfies the RIP.
\end{enumerate}
\end{myth}
\begin{myth}(Ajmal \cite{5})\label{4}

Every isotope of a right B-loop with the LIP is a right B-loop.
\end{myth}

\begin{myexam}
Let $R$ be a ring of all $2\times 2$ matrices taken over the field of three elements and let $G=R\times R$. For all $(u,f),(v,g)\in G$,
define $(u,f)\cdot (v,g)=(u+v,f+g+uv^3)$. Then $(G,\cdot )$ is a loop which is not a right Bol loop but which is a $\sigma$-generalized Bol loop with $\sigma~:x\mapsto x^2$ .
\end{myexam}

We introduce the notions defined below for the first time.
\begin{mydef}\label{a1.1:3}(Twin Special Mappings)

Let $(G,\cdot )$ be a loop and let $\alpha,\beta\in SYM(G)$ such that $\alpha =\psi R_x,\beta=\psi R_y$, for some $x,y\in G$ and $\psi\in SYM(G)$. Then $\alpha$ and $\beta$ are called twin special maps (twins). $\alpha$ (or $\beta$) is called a twin map (twin) of $\beta$ (or $\alpha$) or simply a twin map.
\end{mydef}
Let $(Q,\cdot)$ be a loop. Define
\begin{gather*}
TBS_1(Q,\cdot)=\{\alpha\in SYM(Q)~|~\alpha~\textrm{is any twin map}\},\\
T_1(Q,\cdot)=T_1(Q)=\{\psi\in SYM(Q)~|~\alpha=\psi R_x\in TBS_1(Q,\cdot),~x\in Q,~\psi:e\mapsto e\},\\
TBS_2(Q,\cdot)=\{\alpha\in BS(Q,\cdot)~|~\alpha\in TBS_1(Q,\cdot)\},\\
T_2(Q,\cdot)=T_2(Q)=\{\psi\in SYM(Q)~|~\alpha=\psi R_x\in TBS_2(Q,\cdot),~x\in Q,~\psi:e\mapsto e\}~\textrm{and}\\
T_3(Q,\cdot)=T_3(Q)=\{\psi\in T_2(Q)~|~\alpha^{-1}\sim\beta^{-1}~\textrm{for any twin maps $\alpha,\beta\in SYM(Q)$}\}.
\end{gather*}
Define a relation $\sim$ on $SYM(Q)$ as $\alpha\sim\beta$ if there exists $x\in Q$ such that $\alpha^{-1}=R_x\beta^{-1}$.

The following results will be of judicious use to prove our main results.
\begin{mylem}\label{alpha1} (Bruck \cite{7})

$(H,\circ)$ is a RIPL if and only if $(Q,\cdot)$ is a RIPL.
\end{mylem}
\begin{mylem}\label{alpha2}(Adeniran \cite{phd111})

$(Q,\cdot)$ is a $\sigma$-generalised Bol loop if and only if
$(R_{x}^{-1}, L_xR_{\sigma(x)}, R_{\sigma(x)})\in AUT(Q,\cdot)$ for all $x\in Q$.
\end{mylem}

\begin{mylem}\label{alpha3}(Bruck \cite{8})

Let $(Q,\cdot)$ be a RIPL. If $(U,V,W)\in AUT(Q,\cdot)$, then $(W,JVJ,U)\in AUT(Q,\cdot)$.
\end{mylem}

\section{Main Results}
\begin{myth}\label{4.1}
Let $(Q,\cdot)$ be a loop with a self-map $\sigma$ and let $(H,\circ )$ be the A-holomorph of $(Q,\cdot)$ with a self-map $\sigma'$ such that $\sigma'~:~(\alpha,x)\mapsto (\alpha,\sigma(x))$ for all $(\alpha,x)\in H$. The A-holomorph $(H,\circ)$ of $(Q,\cdot)$ is a $\sigma'$-generalised
Bol loop if and only if $\Big(R_{x}^{-1}, L_xR_{[\sigma(x\gamma^{-1})]\alpha^{-1}},R_{[\sigma(x\gamma^{-1})]\alpha^{-1}}\Big)\in AUT(Q,\cdot)$
for all $x\in Q$ and all $\alpha,\gamma\in A(Q)$.
\end{myth}
\begin{proof}
Define $\sigma'~:H\to H$ as $\sigma'(\alpha,x)=(\alpha,\sigma(x))$. Let $(\alpha,x),(\beta,y),(\gamma,z)\in H$, then by Lemma~\ref{alpha1} and Lemma~\ref{alpha2}, $(H,\circ)$ is a $\sigma'$-generalised Bol loop if and only if
$(\mathbb{R}_{(\alpha,x)^{-1}}, \mathbb{L}_{(\alpha,x)}\mathbb{R}_{\sigma'(\alpha,x)}, \mathbb{R}_{\sigma'(\alpha,x)})\in AUT(H,\circ)$ for all $(\alpha,x)\in H$ if and only if $(\mathbb{R}_{(\alpha,x)^{-1}}, \mathbb{L}_{(\alpha,x)}\mathbb{R}_{(\alpha,\sigma (x))}, \mathbb{R}_{(\alpha,\sigma (x))})\in AUT(H,\circ )\Longleftrightarrow$
\begin{gather}
(\beta,y)\mathbb{R}_{(\alpha,x)^{-1}}\circ(\gamma,z)\mathbb{L}_{(\alpha,x)}\mathbb{R}_{(\alpha,\sigma (x))}
= [(\beta,y)
\circ(\gamma,z)]\mathbb{R}_{(\alpha,\sigma (x))}\label{eq:24}\\
\Leftrightarrow[(\beta,y)\circ(\alpha,x)^{-1}]\circ [((\alpha,x)\circ(\gamma,z))\circ(\alpha,\sigma(x))]
=[(\beta,y)\circ(\gamma,z)]\circ
(\alpha,\sigma(x))\label{eq:25}
\end{gather}
Let $(\beta,y)\circ (\alpha,x)^{-1}=(\tau,t)$. Since $(\alpha,x)^{-1}=(\alpha^{-1},(x\alpha^{-1})^{-1})$, then
\begin{equation}\label{eq:26}
(\tau,t)=(\beta\alpha^{-1},(yx^{-1})\alpha^{-1})
\end{equation}
From \eqref{eq:26},
\begin{gather}
(\tau,t)\circ[(\alpha\gamma,x\gamma\cdot
z)\circ(\alpha,\sigma(x))]=(\beta
\gamma,y\gamma\cdot z)\circ(\alpha,\sigma(x))\label{eq:27}\\
\Leftrightarrow\Big(\tau\alpha\gamma\alpha,(t\alpha\gamma\alpha)\big((x\gamma\cdot
z)\alpha\cdot\sigma(x)\big)\Big)
=\big(\beta\gamma\alpha,(y\gamma\cdot z)\alpha\cdot\sigma(x)\big)\label{eq:28}
\end{gather}
Putting \eqref{eq:26} into \eqref{eq:28}, we have
\begin{gather}
\Big(\beta\alpha^{-1}\alpha\gamma\alpha,(yx^{-1})\alpha^{-1}(\alpha\gamma\alpha)\big((x\gamma\cdot
z)\alpha\cdot\sigma(x)\big)\Big) = \big(\beta\gamma\alpha,(y\gamma\cdot
z)\alpha\cdot\sigma(x)\big)\label{eq:29}\\
\Leftrightarrow\Big(\beta\gamma\alpha, (yx^{-1})\gamma\alpha\big[(x\gamma\cdot z)\alpha\cdot\sigma(x)\big]\Big)
=\big(\beta\gamma\alpha,(y\gamma\cdot z)\alpha\cdot\sigma(x)\big)\label{eq:30}\\
\Leftrightarrow(yx^{-1})\gamma\alpha\cdot [(x\gamma\cdot z)\alpha\cdot\sigma(x)]=(y\gamma\cdot z)\alpha\cdot\sigma(x)\label{eq:31}\\
\Leftrightarrow \big[(yx^{-1})\gamma\cdot [(x\gamma\cdot z)\cdot
(\sigma(x)\alpha^{-1})]\big]\alpha=[(y\gamma\cdot z)
\cdot(\sigma(x)\alpha^{-1})]\alpha\label{eq:32}\\
\Leftrightarrow (y\gamma x^{-1}\gamma )[(x\gamma\cdot z)\cdot (\sigma(x)\alpha^{-1})]=(y\gamma\cdot z)(\sigma(x)\alpha^{-1})\label{eq:33}
\end{gather}
Let $\bar{y}=y\gamma$, then \eqref{eq:33} becomes
\begin{gather}
(\bar{y}\cdot x^{-1}\gamma)[(x\gamma\cdot z)(\sigma(x)\alpha^{-1})]=(\bar{y}\cdot z)(\sigma(x)\alpha^{-1})\label{eq:34}\\
\Leftrightarrow\Big(R_{x\gamma}^{-1}, L_{x\gamma}R_{[\sigma(x)\alpha^{-1}]},R_{[\sigma(x)\alpha^{-1}]}\Big)\in AUT(Q,\cdot)\label{eq:34.}\\
\textrm{and replacing $x\gamma$ by $x$,}~\Big(R_{x}^{-1}, L_xR_{[\sigma(x\gamma^{-1})]\alpha^{-1}},R_{[\sigma(x\gamma^{-1})]\alpha^{-1}}\Big)\in AUT(Q,\cdot).\qed\label{eq:35}
\end{gather}
\end{proof}

\begin{mycor}\label{4.2}
Let $(Q,\cdot)$ be a loop with a self-map $\sigma$
 and let $(H,\circ )$ be the A-holomorph of $(Q,\cdot)$ with a self-map $\sigma'$ such that $\sigma'~:~(\alpha,x)\mapsto (\alpha,\sigma(x))$ for all $(\alpha,x)\in H$. $(H,\circ)$ is a $\sigma'$-generalised
Bol loop if and only if $(Q,\cdot)$ is a $\alpha^{-1}\sigma\gamma^{-1}$-generalised Bol loop for any $\alpha,\gamma\in A(Q)$.
\end{mycor}
\begin{proof}
From Theorem~\ref{4.1}, $(H,\circ)$ is a $\sigma'$-generalised
Bol loop if and only if
$$\Big(R_{x}^{-1}, L_xR_{[\sigma(x\gamma^{-1})]\alpha^{-1}},R_{[\sigma(x\gamma^{-1})]\alpha^{-1}}\Big)\in AUT(Q,\cdot)\Leftrightarrow \Big(R_{x}^{-1}, L_xR_{\sigma''(x)},R_{\sigma''(x)}\Big)\in AUT(Q,\cdot)$$
 where $\sigma''=\alpha^{-1}\sigma\gamma^{-1}$, for all $x\in Q$ and all $\alpha,\gamma\in A(Q)$. It is equivalent to the fact that $(Q,\cdot)$ is a $\sigma''$-generalised Bol loop.\qed
 \end{proof}

\begin{myth}\label{5}
Let $(Q,\cdot)$ be a loop with a self-map $\sigma$ and let $(H,\circ )$ be the holomorph of $(Q,\cdot)$ with a self-map $\sigma'$ such that $\sigma'~:~(\alpha,x)\mapsto (\alpha,\alpha\sigma\gamma (x))$ for all $(\alpha,x)\in H$. Then, $(Q,\cdot)$ is a $\sigma$-generalised Bol
loop if and only if $(H,\circ)$ is a $\sigma'$-generalised
Bol loop.
\end{myth}
\begin{proof}
The proof of this follows from the proof of Theorem~\ref{4.1}.
\end{proof}

\begin{myth}\label{6}
Let $(Q,\cdot)$ be a loop with a self-map $\sigma$ and let $(H,\circ )$ be the holomorph of $(Q,\cdot)$ with a self-map $\sigma'$ such that $\sigma'~:~(\alpha,x)\mapsto (\alpha,\sigma(x))$ for all $(\alpha,x)\in H$. Then, for any $\gamma\in A(Q)$, $(Q,\cdot)$ is a $\sigma\alpha\gamma^{-1}$-generalised flexible-Bol loop if and only if $(H,\circ)$ is a $\sigma'$-generalised
flexible-Bol loop.
\end{myth}
\begin{proof}
\begin{gather}
(R_{x}^{-1}, L_xR_{\sigma(x)}, R_{\sigma(x)})^{-1}=(R_{x},L_x^{-1}R_{\sigma(x)}^{-1}, R_{\sigma(x)}^{-1})\label{eq:35.1}\\
\Leftrightarrow L_x^{-1}R_{\sigma(x)}^{-1}=(L_xR_{\sigma(x)})^{-1}\Leftrightarrow R_{\sigma(x)}L_x=L_xR_{\sigma(x)}\label{eq:35.2}\\
\Leftrightarrow xy\cdot \sigma(x)=x\cdot y\sigma(x)\label{eq:36}
\end{gather}
Let $(\alpha,x),(\beta,y),(\gamma,z)\in H$, then by Lemma~\ref{alpha1} and Lemma~\ref{alpha2},
$(H,\circ)$ is a $\sigma'$-generalised Bol loop if and only if
$(\mathbb{R}_{(\alpha,x)^{-1}}, \mathbb{L}_{(\alpha,x)}\mathbb{R}_{\sigma'(\alpha,x)}, \mathbb{R}_{\sigma'(\alpha,x)})\in AUT(H,\circ)$ for all $(\alpha,x)\in H$. Thus, following \eqref{eq:35.1} to \eqref{eq:36},
\begin{gather}
(\mathbb{R}_{(\alpha,x)}^{-1}, \mathbb{L}_{(\alpha,x)}\mathbb{R}_{\sigma'(\alpha,x)}, \mathbb{R}_{\sigma'(\alpha,x)})^{-1}
=(\mathbb{R}_{(\alpha,x)}, \mathbb{L}_{(\alpha,x)}^{-1}\mathbb{R}_{\sigma'(\alpha,x)}^{-1}, \mathbb{R}_{\sigma'(\alpha,x)}^{-1})\label{eq:37}\\
\Leftrightarrow \mathbb{L}_{(\alpha,x)}\mathbb{R}_{\sigma'(\alpha,x)}=\mathbb{R}_{\sigma'(\alpha,x)}\mathbb{L}_{(\alpha,x)}\Leftrightarrow
(\alpha,x)(\beta,y)\circ \sigma'(\alpha,x)=(\alpha,x)\circ (\beta,y)\sigma'(\alpha,x)\label{eq:37.1}\\
\Leftrightarrow (\alpha,x)(\beta,y)\circ (\alpha,\sigma x)=(\alpha,x)\circ (\beta,y)(\alpha,\sigma x)\\
\Leftrightarrow\big(\alpha\beta\alpha ,(x\beta\cdot y)\alpha\cdot \sigma (x)\big)=\big(\alpha\beta\alpha ,x\beta\alpha\cdot (y\alpha\cdot \sigma (x))\big)\\
\Leftrightarrow (x\beta\alpha\cdot y\alpha)\sigma (x)=x\beta\alpha\cdot(y\alpha\cdot\sigma (x))\\
\Leftrightarrow (x\gamma^{-1}\alpha\beta\alpha\cdot y)\sigma \big(x\gamma^{-1}\alpha\big)=x\gamma^{-1}\alpha\beta\alpha\cdot\big(y\cdot\sigma \big(x\gamma^{-1}\alpha\big)\big)\\
\Leftrightarrow (x\gamma^{-1}\alpha\beta\alpha\cdot y)\sigma\alpha\gamma^{-1}(x)=x\gamma^{-1}\alpha\beta\alpha\cdot(y\cdot\sigma\alpha\gamma^{-1}(x))\\
\Leftrightarrow (xy)\sigma\alpha\gamma^{-1}\big((x(\alpha\beta\alpha)^{-1}\gamma\big)=x\cdot(y\cdot\sigma\alpha\gamma^{-1}\big((x(\alpha\beta\alpha)^{-1}\gamma\big)\\
\Leftrightarrow (xy)\sigma\alpha\gamma^{-1}(x\delta)=x\cdot(y\cdot\sigma\alpha\gamma^{-1}(x\delta)\label{eq:38}
\end{gather}
where $\delta=(\alpha\beta\alpha)^{-1}\gamma\in A(Q)$. So, by \eqref{eq:37.1} to \eqref{eq:38}, $(H,\circ)$ is $\sigma'$-flexible if and only if $(Q,\cdot)$ is $\sigma\alpha\gamma^{-1}$-flexible.

Now, following \eqref{eq:37}, $(H,\circ)$ is a $\sigma'$-generalised
Bol loop if and only if
\begin{gather}
(\mathbb{R}_{(\alpha,x)}, \mathbb{L}_{(\alpha,x)}^{-1}\mathbb{R}_{\sigma'(\alpha,x)}^{-1}, \mathbb{R}_{\sigma'(\alpha,x)}^{-1})\in AUT(Q,\cdot)\\
\Leftrightarrow(\beta,y)\mathbb{R}_{(\alpha,x)}\circ (\gamma, z)\mathbb{L}_{(\alpha,x)}^{-1}\mathbb{R}_{\sigma'(\alpha,x)}^{-1}=[(\beta,y)\circ(\gamma, z)]\mathbb{R}_{\sigma'(\alpha,x)}^{-1}\label{eq:39}
\end{gather}
Let $(\gamma, z)\mathbb{L}_{(\alpha,x)}^{-1}\mathbb{R}_{\sigma'(\alpha,x)}^{-1}=(\mu,u)$ in \eqref{eq:39}, then
$(\gamma,z)=\big(\alpha\mu\alpha,x\mu\alpha(u\alpha\cdot\sigma(x))\big)\Rightarrow \gamma=\alpha\mu\alpha$ and
$z=x\mu\alpha(u\alpha\cdot\sigma(x))$. Consequently,
\begin{equation}\label{eq:40}
\mu=\alpha^{-1}\gamma\alpha^{-1}~\textrm{and}~u=zL_{(x\alpha^{-1}\gamma)}^{-1}R_{(\sigma(x))^{-1}}\alpha^{-1}
=\Big[\big((x\alpha^{-1}\gamma)\backslash z\big)(\sigma(x))^{-1}\Big]\alpha^{-1}
\end{equation}
Also, if $[(\beta,y)\circ(\gamma, z)]\mathbb{R}_{\sigma'(\alpha,x)}^{-1}=(\beta\gamma,y\gamma\cdot z)\mathbb{R}_{(\alpha,\sigma(x))^{-1}}=(\tau,v)$
in \eqref{eq:39}, then
\begin{equation}\label{eq:41}
(\tau,v)=\Big(\beta\gamma\alpha^{-1},(y\gamma\cdot z)\alpha^{-1}\cdot\big((\sigma(x))^{-1}\big)\alpha^{-1}\Big)
\end{equation}
Substituting \eqref{eq:40} and \eqref{eq:41} into \eqref{eq:39}, we get
\begin{gather}
[(\beta,y)\circ (\alpha,x)]\circ(\mu,u)= (\tau,v)\Leftrightarrow
(\beta\alpha\mu, (y\alpha\cdot x)\mu\cdot u)= (\tau,v)\\
\Leftrightarrow
\Big(\beta\gamma\alpha^{-1},(y\alpha\cdot
x)\alpha^{-1}\gamma\alpha^{-1}\cdot \Big([(x\alpha^{-1}\gamma)\backslash z](\sigma(x))^{-1}\Big)\alpha^{-1}\Big)\\=
\Big(\beta\gamma\alpha^{-1},(y\gamma\cdot z)\alpha^{-1}\cdot\alpha^{-1}\big((\sigma(x))^{-1}\big)\Big)\\
\Leftrightarrow\Big\{(y\alpha\cdot
x)\alpha^{-1}\gamma\cdot \Big([(x\alpha^{-1}\gamma)\backslash z](\sigma(x))^{-1}\Big)\Big\}\alpha^{-1}=
\Big[(y\gamma\cdot z)(\sigma(x))^{-1}\Big]\alpha^{-1}\\
\Leftrightarrow(y\alpha\cdot
x)\alpha^{-1}\gamma\cdot \Big([(x\alpha^{-1}\gamma)\backslash z](\sigma(x))^{-1}\Big)=
(y\gamma\cdot z)(\sigma(x))^{-1}\\
\Leftrightarrow (y\gamma\cdot
x\alpha^{-1}\gamma)\cdot \Big([(x\alpha^{-1}\gamma)\backslash z](\sigma(x))^{-1}\Big)=
(y\gamma\cdot z)(\sigma(x))^{-1}\\
\Leftrightarrow  \bar{y}R_{\bar{x}}\cdot zL_{\bar{x}}^{-1}R_{[\sigma(\bar{x}\gamma^{-1}\alpha)]}^{-1}=
(\bar{y}z)R_{[\sigma(\bar{x}\gamma^{-1}\alpha)]}^{-1}\\
\Leftrightarrow \Big(R_{\bar{x}},L_{\bar{x}}^{-1}R_{[\sigma(\bar{x}\gamma^{-1}\alpha)]}^{-1},
R_{[\sigma(\bar{x}\gamma^{-1}\alpha)]}^{-1}\Big)\in AUT(Q,\cdot)\label{eq:42}
\end{gather}
where $\bar{y}=y\gamma$ and $\bar{x}=x\alpha^{-1}\gamma$.
Based on \eqref{eq:38} and the reverse of the procedure from \eqref{eq:35.1} to \eqref{eq:36}, \eqref{eq:42} is true if and only if $(Q,\cdot)$ is a $\sigma\alpha\gamma^{-1}$-generalised Bol
loop.

$\therefore~(Q,\cdot)$ is a $\sigma\alpha\gamma^{-1}$-generalised flexible-Bol
loop if and only if $(H,\circ)$ is a $\sigma'$-generalised
flexible-Bol loop.\qed
\end{proof}

\begin{myth}\label{7}
Let $(Q,\cdot)$ be a loop with a self-map $\sigma$ and let $(H,\circ )$ be the holomorph of $(Q,\cdot)$ with a self-map $\sigma'$ such that $\sigma'~:~(\alpha,x)\mapsto (\alpha,\sigma\gamma\alpha^{-1}(x))$ for all $(\alpha,x)\in H$. Then, for any $\gamma\in A(Q)$, $(Q,\cdot)$ is a $\sigma$-generalised flexible-Bol loop if and only if $(H,\circ)$ is a $\sigma'$-generalised flexible-Bol loop.
\end{myth}
\begin{proof}
The proof of this follows in the sense of Theorem~\ref{6}.
\end{proof}

\begin{myth}\label{a1}
Let $(Q,\cdot)$ be a generalised Bol
loop. If a mapping $\alpha\in BS(Q,\cdot)$ such that $\alpha = \psi R_x$ where $\psi:e\mapsto e$, then $\psi$ is a unique pseudo-automorphism with companion $xg^{-1}\cdot \sigma (x)$ for some $g\in Q$ and for all $x\in Q$.
\end{myth}
\begin{proof}
If $\alpha \in BS(Q,\cdot)$, then $(\alpha
R_g^{-1}, \alpha L_f^{-1},\alpha)\in AUT(Q,\cdot)$ for some $f,g\in
Q$. So, applying Lemma~\ref{alpha3}, $(\alpha,J\alpha
L_f^{-1}J,\alpha R_{g^{-1}})\in AUT(Q,\cdot)$ for
some $f,g\in Q$. Since, $(R_{x^{-1}},
L_xR_{\sigma(x)},R_{\sigma(x)})\in AUT(Q,\cdot)$ for all $x\in Q$, then
\begin{gather}
(\alpha,J\alpha L_f^{-1}J,\alpha
R_{g^{-1}})(R_x^{-1},L_xR_{\sigma(x)},R_{\sigma(x)})=\label{aeq:13.1}\\(\alpha R_{x^{-1}},J\alpha L_f^{-1}J
L_xR_{\sigma(x)},
\alpha R_{g^{-1}}R_{\sigma(x)})\in AUT(Q,\cdot)\label{aeq:13}
\end{gather}
Let $\theta=J\alpha L_f^{-1}JL_xR_{\sigma(x)}$. Then, (\ref{aeq:13}) becomes
\begin{equation}\label{aeq:14}
u\alpha R_{x^{-1}}\cdot v\theta = (u\cdot v)\alpha
R_{g^{-1}}
R_{\sigma(x)}
\end{equation}
for all $u,v\in Q$.
If $\alpha=\psi R_x$, then $\alpha R^{-1}_{x}=\psi$.
Thus, $\theta=J\psi R_xL_f^{-1}JL_xR_{\sigma(x)}$ and (\ref{aeq:14}) becomes
\begin{equation}\label{aeq:15}
u\psi\cdot v\theta = (u\cdot v)\psi R_x
R_{g^{-1}}
R_{\sigma(x)}
\end{equation}
Let $u=e$ in (\ref{aeq:14}), then we have
$e\psi\cdot v\theta = (e\cdot v)\psi R_{x}R_{g^{-1}}R_{\sigma(x)}\Longrightarrow$
\begin{equation}\label{aeq:16}
\theta =\psi R_xR_{g^{-1}}R_{\sigma(x)}
\end{equation}
So by (\ref{aeq:15}) and (\ref{aeq:14}), (\ref{aeq:13}) becomes
$$(\psi,\theta,\psi R_xR_{g^{-1}}R_{\sigma(x)})= \langle\psi,\psi R_xR_{g^{-1}}R_{\sigma(x)},\psi
R_xR_{g^{-1}}R_{\sigma(x)})\in AUT(Q,\cdot)$$ for all $x\in Q$ and some $g\in Q$.
Since $(Q,\cdot)$ is a generalised Bol loop,
$R_xR_{g^{-1}}R_{\sigma(x)}= R_{xg^{-1}\cdot\sigma(x)}$.
Hence,
\begin{equation}\label{aeq:17}
(\psi,\psi R_{xg^{-1}\cdot \sigma(x)},\psi
R_{xg^{-1}\cdot \sigma(x)})\in AUT(Q,\cdot).
\end{equation}
for all $x\in Q$ and some $g\in Q$.
Thus, $\psi$ is a pseudo-automorphism with a companion
$xg^{-1}\sigma(x)$.  \\

Let $\psi_1R_{x_1}=\psi_2R_{x_2}$ where $\psi_1,\psi_2:e\mapsto e$ and $x_1,x_2\in Q$. Then,
$R_{x_1}R^{-1}_{x_2} = \psi^{-1}_1\psi_2$. So,
$eR_{x_1}R^{-1}_{x_2}=e\psi^{-1}_{1}\psi_2$, thus, $x_1x^{-1}_2=e$.
Hence $x_1=x_2$, so $\psi_1=\psi_2$. And this implies
that for all $x\in Q$, there exists a unique $\psi$ such that $\alpha = \psi R_x$.\\
Therefore, $\alpha = \psi R_x$ if and only if $\psi\in
PS(Q,\cdot)$ with companion $xg^{-1}\cdot\sigma(x)$ for some $g\in Q$ and all $x\in Q$.\qed
\end{proof}

\begin{mycor}\label{a2}
Let $(Q,\cdot)$ be a $\sigma$-generalised Bol
loop with $\sigma~:~x\mapsto (xg^{-1})^{-1}$ for all $x\in Q$ and some $g\in Q$. If a mapping $\alpha\in BS(Q,\cdot)$ such that $\alpha = \psi R_x$ where $\psi:e\mapsto e$, then $\psi\in AUM(Q,\cdot)$ is unique.
\end{mycor}
\begin{proof}
Using (\ref{aeq:17}), $(\psi,\psi R_{xg^{-1}\cdot \sigma(x)},\psi R_{xg^{-1}\cdot \sigma(x)})
=(\psi,\psi R_{xg^{-1}\cdot (xg^{-1})^{-1}},\psi R_{xg^{-1}\cdot (xg^{-1})^{-1}})=(\psi,\psi,\psi)\in AUT(Q,\cdot)$. Thus, $\psi$ is an automorphism of $Q$.\qed
\end{proof}

\begin{myth}\label{a3}
Let $(Q,\cdot)$ be a $\sigma$-generalised Bol
loop in which $\sigma \big(x^{-1}\big)=(\sigma(x))^{-1}$ and $xy\cdot \sigma(x)=x\cdot y\sigma(x)$ for all $x,y\in Q$. If a mapping $\alpha\in BS(Q,\cdot)$ such that $\alpha = \psi R_x^{-1}$ where $\psi:e\mapsto e$, then $\psi$ is a unique pseudo-automorphism with companion $x^{-1}g^{-1}\cdot (\sigma (x))^{-1}$ for some $g\in Q$ and for all $x\in Q$.
\end{myth}
\begin{proof}
Lemma~\ref{alpha1} and Lemma~\ref{alpha2}, $(Q,\cdot)$ is a generalised Bol loop if and only
if $(R_{x^{-1}}, L_xR_{\sigma(x)}, R_{\sigma(x)})\in AUT(Q,\cdot)$
for all $x\in Q$. Since $xy\cdot \sigma(x)=x\cdot y\sigma(x)$, then $(R_{x^{-1}},L_xR_{\sigma(x)},
R_{\sigma(x)})^{-1}=(R_x, L_x^{-1}
R_{(\sigma(x))^{-1}}, R_{(\sigma(x))^{-1}})\in AUT(Q,\cdot)$.
$\alpha\in BS(Q,\cdot)\Longleftrightarrow(\alpha R^{-1}_g , \alpha
L^{-1}_f,\alpha)\in AUT(Q,\cdot)\Longrightarrow (\alpha,J\alpha
L_{f^{-1}}J,\alpha R^{-1}_g)\in AUT(Q,\cdot)$
for some $g,f\in Q$ by Lemma~\ref{alpha3}.
Now, the product
\begin{gather}
(\alpha,J\alpha L_{f}^{-1}J,\alpha R_g^{-1})(R_x,L_{x}^{-1}R_{(\sigma(x))^{-1}},
R_{(\sigma(x))^{-1}}) = \\(\alpha R_x,J\alpha L_{f}^{-1}J
L_{x}^{-1}R_{(\sigma(x))^{-1}},\alpha
R_g^{-1}R_{(\sigma(x))^{-1}})\in AUT(Q,\cdot)\label{aeq:18}
\end{gather}
for all $x\in Q$ and some $g,f\in Q$.
Substituting $\alpha=\psi R_x^{-1}$ into (\ref{aeq:18}), we have
\begin{equation}\label{aeq:19}
(\psi,J\psi R_x^{-1}L_f^{-1}J
L_x^{-1}R_{(\sigma(x))^{-1}},
\psi R_x^{-1}R_g^{-1}R_{(\sigma(x))^{-1}})\in AUT(Q,\cdot)
\end{equation}
for all $x\in Q$ and some $g\in Q$.
Now, for all $y,z\in Q$
\begin{equation}\label{aeq:20}
y\psi\cdot zJ\psi R_x^{-1}L_f^{-1}J
L_x^{-1}R_{(\sigma(x))^{-1}}=(yz)\psi R_x^{-1}R_g^{-1}R_{(\sigma(x))^{-1}}
\end{equation}
Putting $y=e$ in (\ref{aeq:20}), we have
\begin{equation}\label{aeq:21}
J\psi R_x^{-1}L_f^{-1}J
L_x^{-1}R_{(\sigma(x))^{-1}}
=\psi R_x^{-1}R_g^{-1}R_{(\sigma(x))^{-1}}
\end{equation}
for all $x\in Q$ and some $g\in Q$.
Thus, using (\ref{aeq:21}) in (\ref{aeq:20}),
\begin{equation}\label{aeq:22}
(\psi,\psi R_x^{-1}R_g^{-1}R_{(\sigma(x))^{-1}},
\psi R_x^{-1}R_g^{-1}R_{(\sigma(x))^{-1}})\in AUT(Q,\cdot)
\end{equation}
for all $x\in Q$ and some $g\in Q$.

Since $(Q,\cdot)$ is a
generalised Bol loop,
$$R_{x^{-1}}R_{g^{-1}}R_{(\sigma(x))^{-1}}=R_{x^{-1}g^{-1}\cdot (\sigma(x))^{-1}}.$$
Hence,
\begin{equation}\label{aeq:23}
\big(\psi,\psi R_{x^{-1}g^{-1}\cdot (\sigma(x))^{-1}},
\psi R_{x^{-1}g^{-1}\cdot (\sigma(x))^{-1}}\big)\in AUT(Q,\cdot).
\end{equation}
The proof the uniqueness of $\psi$ is similar to that in Theorem~\ref{a1}.
Therefore, $\psi$ is a unique pseudo-automorphism of $(Q,\cdot)$ with companion
$x^{-1}g^{-1}\cdot(\sigma(x))^{-1}$.\qed
\end{proof}

\begin{mycor}\label{a4}
Let $(Q,\cdot)$ be a $\sigma$-generalised Bol
loop and an AIPL in which $\sigma \big(x^{-1}\big)=(\sigma(x))^{-1}$ and $xy\cdot \sigma(x)=x\cdot y\sigma(x)$ for all $x,y\in Q$ where   $\sigma~:~x\mapsto (xg)^{-1}$ for all $x\in Q$ and some $g\in Q$. If a mapping $\alpha\in BS(Q,\cdot)$ such that $\alpha = \psi R_x^{-1}$ where $\psi:e\mapsto e$, then $\psi\in AUM(Q,\cdot)$ is unique.
\end{mycor}
\begin{proof}
Using \eqref{aeq:23},
\begin{gather*}
\big(\psi,\psi R_{x^{-1}g^{-1}\cdot (\sigma(x))^{-1}},
\psi R_{x^{-1}g^{-1}\cdot (\sigma(x))^{-1}}\big)=\big(\psi,\psi R_{x^{-1}g^{-1}\cdot ((xg)^{-1})^{-1}},\psi R_{x^{-1}g^{-1}\cdot ((xg)^{-1})^{-1}}\big)\\\
=(\psi,\psi,\psi)\in AUT(Q,\cdot).
\end{gather*}
Thus, $\psi$ is an automorphism of $Q$.\qed
\end{proof}

\begin{mylem}\label{a4.2}
Let $(Q,\cdot)$ be a $\sigma$-generalised
Bol loop. Then
\begin{enumerate}
\item $\sim$ is an equivalence relation over $SYM(Q)$.
\item For any $\alpha,\beta\in SYM(Q)$, $\alpha\sim\beta$ if and only if $\alpha,\beta\in TBS_1(Q,\cdot)$.
\item $\displaystyle TBS_1(Q,\cdot)=\bigcup_{[\alpha ]\in SYM(Q)/\sim}[\alpha ]$.
\end{enumerate}
\end{mylem}
\begin{proof}
\begin{enumerate}
\item Let $\alpha,\beta,\gamma\in SYM(Q)$.
With $x=e$, $\alpha^{-1}=R_e\alpha^{-1}$ and so $\alpha\sim\alpha$. Thus, $\sim$ is reflexive. Let $\alpha\sim\beta$, then there exists $x\in Q$ such that $\alpha^{-1}=R_x\beta^{-1}\Longrightarrow\beta^{-1}=R_{x^{-1}}\alpha^{-1}\Longrightarrow\beta\sim\alpha$. Thus, $\sim$ is symmetric. Let $\alpha\sim\beta$ and $\beta\sim\gamma$, then there exist $x,y\in Q$ such that $\alpha^{-1}=R_x\beta^{-1}$ and $\beta^{-1}=R_y\gamma^{-1}\Longrightarrow\alpha^{-1}=R_xR_y\gamma^{-1}$. Choose $y=\sigma (x)$, so that $\alpha^{-1}=R_xR_{\sigma (x)}\gamma^{-1}=R_{x\sigma (x)}\gamma^{-1}\Longrightarrow\alpha\sim\gamma$. $\therefore~\sim$ is an equivalence relation over $SYM(Q)$.
\item Let $\alpha,\beta\in SYM(Q)$. Let $\alpha\sim\beta$, then there exists $y\in Q$ such that $\alpha^{-1}=R_y\beta^{-1}$. Take $y=x\sigma(x)$, then $\alpha^{-1}=R_{x\sigma(x)}\beta^{-1}=R_xR_{\sigma(x)}\beta^{-1}\Rightarrow\alpha R_x=\beta R_{\sigma(x)^{-1}}$.
    Say, $\alpha R_x=\beta R_{\sigma(x)^{-1}}=\psi$, then $\alpha =\psi R_{x^{-1}}$ and $\beta =\psi R_{\sigma(x)}$. So, $\alpha,\beta\in TBS_1(Q,\cdot)$.

    Let $\alpha,\beta\in TBS_1(Q,\cdot)$. Then there exist $x,y\in Q$, $\psi\in SYM(Q)$ such that $\alpha =\psi R_x$ and $\beta =\psi R_y$. This implies $\psi =\alpha R_x^{-1}=\beta R_y^{-1}\Rightarrow \alpha^{-1}=R_{x^{-1}}R_y\beta^{-1}$. Take $y=\sigma(x^{-1})$, then
    $\alpha^{-1}=R_{x^{-1}}R_{\sigma(x^{-1})}\beta^{-1}=R_{x^{-1}\sigma(x^{-1})}\beta^{-1}\Rightarrow\alpha\sim\beta$.
    \item Use 1. and 2.\qed
\end{enumerate}
\end{proof}

\begin{mylem}\label{a4.1}
Let $(Q,\cdot)$ be a loop. Then
\begin{enumerate}
\item $TBS_1(Q,\cdot)\le SYM(Q)$ if and only if $\alpha^{-1}\sim\beta^{-1}$ for any twin maps $\alpha,\beta\in SYM(Q)$.
Hence, $T_1(Q,\cdot)\le SYM(Q)$.
\item $TBS_2(Q,\cdot)\le BS(Q,\cdot )$ if and only if $\alpha^{-1}\sim\beta^{-1}$ for any twin maps $\alpha,\beta\in SYM(Q)$.
Hence, $T_2(Q,\cdot)\le PS(Q,\cdot )$.
\end{enumerate}
\end{mylem}
\begin{proof}
\begin{enumerate}
\item $TBS_1(Q,\cdot)\ne\emptyset$ because $I=IR_e$ and $I^{-1}=I^{-1}R_e$ and so, $I,I^{-1}\in TBS_1(Q,\cdot)$.
Let $\alpha_1,\alpha_2\in TBS_1(Q,\cdot)$ and let $\psi_1,\psi_2\in SYM(Q)$. Then, there exist $x_1,y_1,x_2,y_2\in Q$, $\psi_1,\psi_2\in SYM(Q)$ and  $\beta_1,\beta_2\in SYM(Q)$ such that $\alpha_1=\psi_1R_{x_1},~\beta_1=\psi_1R_{y_1}$ and $\alpha_2=\psi_2R_{x_2},~\beta_2=\psi_2R_{y_2}$. So, $\alpha_1\alpha_2^{-1}=\psi_1R_{x_1}R_{x_2}^{-1}\psi_2^{-1}$.
Now, $\alpha_1\alpha_2^{-1}\in TBS_1(Q,\cdot)\Leftrightarrow \alpha_1\alpha_2^{-1}=\psi R_x$ and $\beta_1\beta_2^{-1}=\psi R_y$ for some $x,y\in Q$ and $\psi\in SYM(Q)$. Taking $\psi =\psi_1\psi_2^{-1}$ and $x=x_2$, then $\alpha_1\alpha_2^{-1}=\psi_1\psi_2^{-1}R_x\Leftrightarrow \psi_1R_{x_1}R_{x_2}^{-1}\psi_2^{-1}=\psi_1\psi_2^{-1}R_x\Leftrightarrow \psi_2R_{x_1}=R_x\psi_2R_{x_2}\Leftrightarrow\psi_2R_{x_1}=R_x\alpha_2\Leftrightarrow\psi_2R_{y_2}=R_x\alpha_2~\textrm{with}~x_1=y_2\Leftrightarrow
\beta_2=R_x\alpha_2\Leftrightarrow\alpha_2^{-1}\sim\beta_2^{-1}$. Thus, $TBS_1(Q,\cdot)\le SYM(Q)$ if and only if $\alpha_2^{-1}\sim\beta_2^{-1}$.

Assuming that $TBS_1(Q,\cdot)\le SYM(Q)$, then $T_1(Q,\cdot)\ne\emptyset$ because $I\in T_1(Q,\cdot)$. As earlier shown, $\alpha_1\alpha_2^{-1}=\psi_1\psi_2^{-1}R_x$ for any $\psi_1,\psi_2\in T_1(Q,\cdot)$ and $\alpha_1,\alpha_2\in TBS_1(Q,\cdot)$. So, $T_1(Q,\cdot)\le SYM(Q)$.
\item $TBS_2(Q,\cdot)\ne\emptyset$ because $TBS_1(Q,\cdot)\ne\emptyset$ and $BS(Q,\cdot)\ne\emptyset$. For any $\alpha_1,\alpha_2\in TBS_2(Q,\cdot)$, $\alpha_1\alpha_2^{-1}\in BS(Q,\cdot)$. So, $\alpha_1\alpha_2^{-1}\in TBS_2(Q,\cdot)\Leftrightarrow \alpha_1\alpha_2^{-1}\in TBS_1(Q,\cdot)\Leftrightarrow\alpha_2^{-1}\sim\beta_2^{-1}$. $\therefore~TBS_2(Q,\cdot)\le BS(Q,\cdot )\Leftrightarrow\alpha_2^{-1}\sim\beta_2^{-1}$.

Assuming that $TBS_2(Q,\cdot)\le BS(Q,\cdot )$, then $T_2(Q,\cdot)\ne\emptyset$ because $I\in T_2(Q,\cdot)$.
Let $\psi\in T_2(Q,\cdot)$, then there exists $\alpha\in  BS(Q,\cdot )$, and $\alpha =\psi R_x\in TBS_2(Q,\cdot)$ for some $x\in Q$.
Recall that $\alpha\in  BS(Q,\cdot )$ implies there exist $f,g\in Q$ such that $(\alpha R_g^{-1},\alpha L_f^{-1},\alpha )\in AUT(Q,\cdot )$.
Taking $g=x$ and $f=e$, $(\alpha R_g^{-1},\alpha L_f^{-1},\alpha )=(\psi R_xR_x^{-1},\psi R_xL_e^{-1},\psi R_x )=(\psi,\psi R_x,\psi R_x )\in AUT(Q,\cdot )\Rightarrow\alpha\in PS(Q,\cdot )$. Thus, $T_2(Q,\cdot)\subseteq PS(Q,\cdot )$.

Let $\psi_1,\psi_2\in T_2(Q,\cdot)$, then there exist $\alpha_1,\alpha_2\in TBS_2(Q,\cdot)$ such that $\alpha_1=\psi_1R_{x_1}$ and $\alpha_2=\psi_2R_{x_2}$. In fact, $\alpha_1,\alpha_2\in TBS_1(Q,\cdot)$ and so, following 1., $\alpha_1\alpha_2^{-1}=\psi_1\psi_2^{-1}R_y\in TBS_1(Q,\cdot)$ for some $y\in Q$. This implies that $\alpha_1\alpha_2^{-1}=\psi_1\psi_2^{-1}R_y\in TBS_2(Q,\cdot)$ for some $y\in Q$ and so $\psi_1\psi_2^{-1}\in T_2(Q,\cdot)$. Thus, $T_2(Q,\cdot)\le PS(Q,\cdot )$.\qed
\end{enumerate}
\end{proof}

\paragraph{}In what follows, in a loop $(Q,\cdot)$ with A-holomorph $(H,\circ )$ where $H=A(Q)\times Q$, we shall replace $A(Q)$ by $T_3(Q)$ whenever $T_3(Q)\le AUM(Q,\cdot)$ and then call $(H,\circ )$ a T${}_3$-holomorph of $(Q,\cdot)$.
\begin{mycor}\label{t1}
Let $(Q,\cdot)$ be a loop with a self-map $\sigma~:~x\mapsto (xg^{-1})^{-1}$ for all $x\in Q$ and some $g\in Q$
 and let $(H,\circ )$ be the T${}_3$-holomorph of $(Q,\cdot)$ with a self-map $\sigma'$ such that $\sigma'~:~(\alpha,x)\mapsto \big(\alpha,(xg^{-1})^{-1}\big)$ for all $(\alpha,x)\in H$. Then, $(H,\circ)$ is a $\sigma'$-generalised
Bol loop if $(Q,\cdot)$ is a $\alpha^{-1}\sigma\gamma^{-1}$-generalised Bol loop for any $\alpha,\gamma\in T_3$.
\end{mycor}
\begin{proof}
This is proved with Lemma~\ref{a4.1}, Corollary~\ref{4.2} and Corollary~\ref{a2}.
\qed
\end{proof}

\begin{mycor}\label{t2}
Let $(Q,\cdot)$ be a loop with a self-map $\sigma~:~x\mapsto (xg^{-1})^{-1}$ for all $x\in Q$ and some $g\in Q$
 and let $(H,\circ )$ be the T${}_3$-holomorph of $(Q,\cdot)$ with a self-map $\sigma'$ such that $\sigma'~:~(\alpha,x)\mapsto \Big(\alpha,\big[\alpha\gamma(x)(\alpha(g))^{-1}\big]^{-1}\Big)$ for all $(\alpha,x)\in H$ and any $\gamma\in T_3$.
If $(Q,\cdot)$ is a $\sigma$-generalised Bol
loop, then $(H,\circ)$ is a $\sigma'$-generalised Bol loop.
\end{mycor}
\begin{proof}
This is proved with Lemma~\ref{a4.1}, Theorem~\ref{5} and Corollary~\ref{a2}.
\qed
\end{proof}

\begin{mycor}\label{t3}
Let $(Q,\cdot)$ be a loop with a self-map $\sigma~:~x\mapsto (xg^{-1})^{-1}$ for all $x\in Q$ and some $g\in Q$
 and let $(H,\circ )$ be the T${}_3$-holomorph of $(Q,\cdot)$ with a self-map $\sigma'$ such that $\sigma'~:~(\alpha,x)\mapsto \big(\alpha,(xg^{-1})^{-1}\big)$ for all $(\alpha,x)\in H$. If for any $\gamma\in T_3$, $(Q,\cdot)$ is a $\sigma\alpha\gamma^{-1}$-generalised flexible-Bol loop, then $(H,\circ)$ is a $\sigma'$-generalised
flexible-Bol loop.
\end{mycor}
\begin{proof}
This is proved with Lemma~\ref{a4.1}, Theorem~\ref{6} and Corollary~\ref{a2}.
\qed
\end{proof}

\begin{mycor}\label{t4}
Let $(Q,\cdot)$ be a loop with a self-map $\sigma~:~x\mapsto (xg^{-1})^{-1}$ for all $x\in Q$ and some $g\in Q$
 and let $(H,\circ )$ be the T${}_3$-holomorph of $(Q,\cdot)$ with a self-map $\sigma'$ such that $\sigma'~:~(\alpha,x)\mapsto \Big(\alpha,\big[\big(\gamma\alpha^{-1}(x)\big)g^{-1}\big]^{-1}\Big)$ for all $(\alpha,x)\in H$ and any $\gamma\in T_3$.
If $(Q,\cdot)$ is a $\sigma$-generalised flexible-Bol loop, then $(H,\circ)$ is a $\sigma'$-generalised flexible-Bol loop.
 \end{mycor}
\begin{proof}
This is proved with Lemma~\ref{a4.1}, Theorem~\ref{7} and Corollary~\ref{a2}.
\qed
\end{proof}

\begin{myrem}
In Corollary~\ref{t1},~\ref{t2},~\ref{t3},~\ref{t4}, the holomorph of a loop is built on the group of automorphisms gotten via the group of twin mappings.
\end{myrem}

\end{document}